\documentclass[a4paper]{amsart}

\usepackage{amsmath,amsthm,amssymb,latexsym,epic,bbm,comment, mathrsfs}
\usepackage{graphicx,enumerate,stmaryrd,color}
\usepackage[all,2cell]{xy}
\xyoption{2cell}
\usepackage[all]{xy}
\usepackage[active]{srcltx}
\usepackage[parfill]{parskip}

\newtheorem{theorem}{Theorem}
\newtheorem{lemma}{Lemma}

\newtheorem{proposition}{Proposition}

\newcommand{\tto}{\twoheadrightarrow}

\begin{document}

\title[Tilting modules and exceptional sequences]
{Tilting modules and exceptional sequences\\ for leaf quotients of type \textit{A} zig-zag algebras}

\author{Elin Persson Westin}

\begin{abstract}
We classify generalized tilting modules and full exceptional sequences for 
the family of quasi-hereditary quotients of type $A$ zig-zag algebras 
and for a related family of algebras. We also give a
characterization of these quotients as quasi-hereditary algebras with simple
preserving duality that are ``close'' to self-injective algebras.
\end{abstract}

\maketitle

\noindent


\noindent


\section{Introduction and description of the results}\label{s1}

Tilting theory, which originates from \cite{BB,HR}, plays an important role in contemporary 
representation theory of finite-dim\-en\-sio\-nal algebras. At the same time, classification
of all tilting modules for a given algebra is usually a hard problem.  
This problem has been studied in various special cases, for example in \cite{BK,MU,Ya}, 
and some generalizations of this problem were studied in \cite{Ad}. 
For self-injective algebras this problem is trivial.
The generalized tilting modules for the Auslander algebra of $\Bbbk[x]/(x^n)$ were classified in \cite{Ge}.
This algebra is, in fact, quadratic dual 
to the main protagonist of the present note. In an earlier work, \cite{BHRR} characterized classical tilting modules for this algebra as the $\Delta$-filtered generalized tilting modules.

In the present note we classify all generalized tilting modules for the quasi-hereditary
quotient  $B_n$ of the type $A$ zig-zag algebra with $n$ simple modules (cf. \cite{KS,HK,ET}). These algebras describe, in particular, blocks of Temperley-Lieb algebras,
cf. \cite{Ma,BRSA}, and some blocks of the parabolic category $\mathcal{O}$
for $\mathfrak{sl}_n$, cf. \cite{St}.
These algebras are very close to self-injective algebras
(in the sense that all indecomposable projective modules but one are injective) and have
finite representation type. Both these factors significantly simplify our arguments.
As it turns out, see Theorem~\ref{thm4}, all generalized tilting modules for such algebras
are obtained by adding a standard or a costandard module (in the terminology of the 
quasi-hereditary structure) to a basic projective-injective module.

Other interesting, but much less studied, objects are so-called {\em exceptional sequences}
of modules. A recent preprint \cite{HP} provides a classification of full exceptional sequences
for the Auslander algebra of $\Bbbk[x]/(x^n)$, which is, as mentioned above, the quadratic dual of our algebra $B_n$.  In Theorem~\ref{thm8} we provide a classification
of full exceptional sequences for $B_n$. We show that each exceptional sequence
for $B_n$ is uniquely determined by choosing one module in each pair 
consisting of a standard and the corresponding costandard $B_n$-module.

Apart from these classification results, in Theorem~\ref{thm2} we characterize $B_n$
as the basic, connected, quasi-hereditary algebra with  $n$ simple modules which has
a simple preserving duality and for which all indecomposable projective modules
but one are injective.

The paper is organized as follows: all preliminaries are collected in Section~\ref{s2}.
Theorem~\ref{thm2} is proved in Section~\ref{s3}. Section~\ref{s4} contains a technical
result on non-vanishing of second self-extensions of certain indecomposable $B_n$-modules.
This result plays an important role both in the proof of Theorem~\ref{thm4}, which can be found in 
Section~\ref{s5}, and in the proof of Theorem~\ref{thm8}, which can be found in 
Section~\ref{s6}. In the last section, Section~\ref{s7}, we also solve similar classification 
problems for a slightly smaller quotient of type $A$ zig-zag algebras.

\section{Zig-zag algebras, their quotients and modules}\label{s2}

\subsection{Basic setup and notation}\label{s2.1}

We work over an algebraically closed field $\Bbbk$. For a finite-dimensional
$\Bbbk$-algebra $A$ given by some quiver
and (admissible) relations, we denote by $L(v)$ the simple $A$-module associated to a vertex 
$v$ of the  quiver. We also denote by $P(v)$ the indecomposable projective cover of $L(v)$
and by $I(v)$ the indecomposable injective envelope of $L(v)$. If not stated
otherwise, by module we mean a right module.

By a {\em simple preserving duality} of a module category we mean an involutive contravariant
autoequivalence which preserves the isomorphism classes of simple modules.
Simple preserving dualities are given by involutive algebra anti-automorphisms which 
fix pointwise some full set of primitive orthogonal idempotents.

We will often abuse notation and write $X=Y$ instead of $X\cong Y$ for two modules $X$ and $Y$.

\subsection{Zig-zag algebras}\label{s2.2}

Let $Q$ be a finite connected unoriented graph without loops and with at least one edge. 
Let $\tilde{Q}$ denote the quiver obtained
from $Q$ via substituting every edge $\xymatrix{i\ar@{-}[r]&j}$ in $Q$ by two oriented edges
$\xymatrix{i\ar@/^/[r]&j\ar@/^/[l]}$. We denote by $A_Q$ the quotient of the path algebra 
$\Bbbk\tilde{Q}$ of $\tilde{Q}$ by the ideal generated by the following relations:
\begin{itemize}
\item any path of length three is zero;

\item any path of length two which is not a cycle is zero;

\item for any vertex $v$, all length two cycles which start and terminate at $v$ are equal.

\end{itemize}

The algebra $A_Q$ is usually called the {\em zig-zag algebra} associated with $Q$, see \cite{HK,ET}.
Directly from the definition it follows that $A_Q$ is $\mathbb{Z}$-graded by path length and this grading is positive
in the sense that all non-zero homogeneous components have non-negative degree and the
degree zero component is semi-simple, see \cite{HK}.

It follows directly from the definition that, for each vertex $v$ in $Q$, the projective module
$P(v)$ has Loewy length three and is rigid in the sense that the radical and the socle filtrations
of $P(v)$ coincide. In particular, $P(v)$ has a unique Loewy filtration. The layers of this
unique Loewy filtration are given as follows:
\begin{itemize}
\item the top of $P(v)$ is isomorphic to $L(v)$;

\item the socle of $P(v)$ is isomorphic to $L(v)$;

\item the quotient $\mathrm{Rad}(P(v))/\mathrm{Soc}(P(v))$ is semisimple and, for any vertex $w$,
the multiplicity of $L(w)$ in this quotient equals the number of edges in $Q$ between $v$ and $w$.
\end{itemize}
Consequently, $P(v)\cong I(v)$ and $A_Q$ is a self-injective algebra.

\subsection{The algebras $A_n$, $B_n$ and $C_n$}\label{s2.3}

For $n\in\mathbb{Z}_{>1}$, we denote by $A_n$ the algebra $A_Q$, where $Q$ is the following 
Dynkin diagram of type $A$:
\begin{displaymath}
\xymatrix{
\mathtt{1}\ar@{-}[r]&\mathtt{2}\ar@{-}[r]&\mathtt{3}\ar@{-}[r]&\cdots\ar@{-}[r]& \mathtt{n}
}
\end{displaymath}
We denote by $B_n$ the quotient of $A_n$ by the additional relation that the length two loop at 
the vertex $\mathtt{n}$ is zero.
We denote by $C_n$ the quotient of $B_n$ by the additional relation that the length two loop at 
the vertex $\mathtt{1}$ is zero.
The vertices $\mathtt{1}$ and $\mathtt{n}$ are both leaves in the Dynkin diagram of type $A$. We therefore refer to these quotients as {\em leaf quotients of type $A$ zig-zag algebras}.

Both $B_n$ and $C_n$ inherit a $\mathbb{Z}$-grading from $A_n$. It is easy to see that 
$B_n$ is never self-injective, while $C_n$ is self-injective if and only if $n=2$.

The algebra $A_n$ has an involutive anti-automorphism given by swapping the arrows in each 
part $\xymatrix{i\ar@/^/[r]&j\ar@/^/[l]}$ from the definition. This induces 
involutive anti-auto\-morphisms on $B_n$ and $C_n$. In particular, module categories over
all these algebras have simple preserving dualities.

\subsection{Indecomposable modules over $A_n$, $B_n$ and $C_n$}\label{s2.4}

By Corollary 1 in \cite{HK} the algebra $A_n$ has $n(n+1)$ isomorphism classes of indecomposable modules. 
It is easiliy checked that the algebra $A_n$ is special biserial in the sense of \cite{BR,WW}. In particular, all indecomposable
$A_n$-modules split into three families: pro\-jective-injective modules, string modules and band modules, see Proposition 2.3 in \cite{WW}.
Even more, in case of the algebra $A_n$ there are no band modules. Apart from the projective-injective modules $P(i)$ and the simple modules $L(i)$, where $i=1,2\dots,n$,  we only have 
four finite families of string modules which look as follows:
\begin{itemize}
\item The modules $M(i,j)$, where $1\leq i<j\leq n$ are such that $i\equiv j\mod 2$:

\begin{displaymath}
\xymatrixcolsep{1pc}\xymatrix{
&L(\mathtt{i+1})\ar[dl]\ar[dr]&&\ar[dl]\dots&&L(\mathtt{j-1})\ar[dr]\ar[dl]&\\
L(\mathtt{i})&&L(\mathtt{i+2})&&\dots&& L(\mathtt{j})
}
\end{displaymath}

\item The modules $N(i,j)$, where $1\leq i<j\leq n$ are such that $i\not\equiv j\mod 2$:
\begin{displaymath}
\xymatrixcolsep{1pc}\xymatrix{
&L(\mathtt{i+1})\ar[dl]\ar[dr]&&\ar[dl]\dots&&L(\mathtt{j})\ar[dl]&\\
L(\mathtt{i})&&L(\mathtt{i+2})&&\dots&&
}
\end{displaymath}
\item The modules $W(i,j)$, where $1\leq i<j\leq n$ are such that $i\equiv j\mod 2$:
\begin{displaymath}
\xymatrixcolsep{1pc}\xymatrix{
L(\mathtt{i})\ar[dr]&&L(\mathtt{i+2})\ar[dl]\ar[dr]&&\dots\ar[dr]&&L(\mathtt{j})\ar[dl]\\
&L(\mathtt{i+1})&&\dots&&L(\mathtt{j-1})&
}
\end{displaymath}
\item The modules $S(i,j)$, where $1\leq i<j\leq n$ are such that $i\not\equiv j\mod 2$:
\begin{displaymath}
\xymatrixcolsep{1pc}\xymatrix{
L(\mathtt{i})\ar[dr]&&L(\mathtt{i+2})\ar[dl]\ar[dr]&&\dots\ar[dr]&&\\
&L(\mathtt{i+1})&&\dots&&L(\mathtt{j})&
}
\end{displaymath}
\end{itemize}
It is clear that all string modules in the four families above are indecomposable and not isomorphic to any projective-injective nor simple module. This completes the classification of the $n(n+1)$ isomorphism classes of indecomposable $A_n$-modules. 

All indecomposable $A_n$-modules but $P(n)$ are also $B_n$-modules.
All indecomposable $B_n$-modules but $P(1)$ are also $C_n$-modules.

\subsection{Quasi-hereditary structure on $B_n$}\label{s2.5}

Let $\Lambda$ be a finite-dimensional $\Bbbk$-algebra with a fixed order $L_1<L_2<\dots<L_k$
on a full set of representatives of simple $\Lambda$-modules. For $i=1,2,\dots,k$, we denote by
$P_i$ the indecomposable projective cover of $L_i$ and by 
$I_i$ the indecomposable injective envelope of $L_i$. Recall, see \cite{DR,CPS}, that 
$\Lambda$ is called {\em quasi-hereditary} (with respect to the order $<$) provided that there
are $\Lambda$-modules $\Delta_i$, for $i=1,2,\dots,k$, (called {\em standard modules}) with the properties that
\begin{itemize}
\item $P_i\tto \Delta_i$ and the kernel of this surjection has a filtration with
subquotients $\Delta_j$, for $j>i$;
\item $\Delta_i\tto L_i$ and the kernel of this surjection has composition 
subquotients $L_j$, for $j<i$.
\end{itemize}
Equivalently, there
are $\Lambda$-modules $\nabla_i$, for $i=1,2,\dots,k$, (called {\em costandard modules}) with the properties that
\begin{itemize}
\item $\nabla_i\hookrightarrow I_i$ and the cokernel of this injection has a filtration with
subquotients $\nabla_j$, for $j>i$;
\item $L_i\hookrightarrow \nabla_i$ and the cokernel of this injection has composition 
subquotients $L_j$, for $j<i$.
\end{itemize}

Following \cite{Ri}, we denote by $T_i$ the indecomposable tilting $\Lambda$-module corresponding to $i$.
It is uniquely determined by the properties that $\Delta_i\hookrightarrow T_i$ and the cokernel of 
this inclusion has a filtration with standard subquotients. Equivalently, $T_i\tto \nabla_i$
and the kernel of this surjection has a filtration with costandard subquotients.

The terminology here is slightly conflicting. We will refer to tilting modules in the sense of \cite{Mi} by \emph{generalized tilting modules} and those that are connected to the quasi-hereditary structure, in the sense of \cite{Ri}, simply by \emph{tilting modules}.

Consider the algebra $\Lambda=B_n$ and choose the order $L(1)<L(2)<\dots<L(n)$.
Set $\Delta(1)=\nabla(1)=L(1)$; $\Delta(i)=S(i-1,i)$, for $i>1$; and
$\nabla(i)=N(i-1,i)$, for $i>1$. Then, for $i<n$, there are obvious short exact sequences
\begin{displaymath}
\Delta(i+1)\hookrightarrow P(i)\tto \Delta(i)\qquad\text{ and }\qquad 
\nabla(i)\hookrightarrow I(i)\tto \nabla(i+1).
\end{displaymath}
Consequently, $B_n$ is quasi-hereditary. It is easy to see that $B_n$ is {\em not} quasi-hereditary
with respect to any other order on the full set of representatives of isomorphism classes of simple
$B_n$-modules. Set 
\begin{displaymath}
\mathbf{D}:=\{\Delta(n),\Delta(n-1),\dots,\Delta(1)=\nabla(1),\nabla(2),\dots,\nabla(n)\}. 
\end{displaymath}

\begin{proposition}\label{cor1}
For a quasi-hereditary algebra $\Lambda$ we have the following vanishing extensions. 
For all $k>0$ and for all $i\leq j$, we have
\begin{displaymath}
\mathrm{Ext}_{\Lambda}^k(\Delta_j,\Delta_i)=0 \text{ and } \mathrm{Ext}_{\Lambda}^k(\nabla_i,\nabla_j)=0.
\end{displaymath}

For all $k>0$ and for all $i,j$, we have
\begin{displaymath}
\mathrm{Ext}_{\Lambda}^k(\Delta_i,\nabla_j)=0.
\end{displaymath}

In particular, for any $X\in \mathbf{D}$ and $k>0$, we have 
\begin{displaymath}
\mathrm{Ext}^k_{B_n}(X,X)=0. 
\end{displaymath}
\end{proposition}

\begin{proof}
Using Lemma 2.2 in \cite{KK}, see also \cite[Theorem 2.3]{CPS}, in the first claim we may assume that $j$ is maximal, so 
$\Delta_j$ is projective and $\nabla_j$ is injective. The claim is then clear. For the second claim, see Proposition~2.1 in \cite{KK} or Corollary~3 in \cite{Ri}.
The third claim is a special case of the first one.

\end{proof}

Note that, for the algebra $B_n$, we have  $T(1)=L(1)$ and $T(i)=P(i-1)=I(i-1)$, for $i>1$.

\section{A characterization of the algebra $B_n$}\label{s3}

In this section we propose a characterization of the quasi-hereditary algebras $B_n$ inside the
class of all quasi-hereditary algebras. Our result says that the algebras $B_n$ are non-semi-simple
quasi-hereditary algebras which are, in some sense, ``closest'' to self-injective algebras.

\begin{theorem}\label{thm2}
Let $\Lambda$ be a basic, connected, quasi-hereditary algebra with respect to some order
$L_1<L_2<\dots<L_n$, where $n\in\mathbb{Z}_{>0}$. Assume the following:
\begin{enumerate}[$($a$)$]
\item\label{thm2.1} Exactly $n-1$ of indecomposable projective $\Lambda$-modules are injective.
\item\label{thm2.2} The algebra $\Lambda$ has a simple preserving duality. 
\end{enumerate}
Then $n>1$ and $\Lambda\cong B_{n}$.
\end{theorem}

\begin{proof}
By \eqref{thm2.2}, $\Lambda$ has a simple preserving duality which we denote by $\star$.
Then $P_i^\star=I_i$, for all $i$.

Our first observation is that the module $P_n$ is not injective. Indeed, assume that $P_n\cong I_j$,
for some $j$. If $j=n$, then $P_n$ has top and socle isomorphic to $L_n$ and this simple module also
appears in $P_n=\Delta_n$ with multiplicity $1$. Therefore $P_n=I_n=L_n$. This means that $P_n$ is a direct 
summand of $\Lambda$ as an algebra. As $\Lambda$ is assumed to be connected, we obtain that $n=1$.
At the same time, the only basic quasi-hereditary 
algebra with one isomorphism class of simple modules is $\Bbbk$ and this algebra does not satisfy 
\eqref{thm2.1} as there is at least one projective-injective module. Therefore $P_n=I_n$ is not possible. As a bonus, we now know that $n>1$.

Assume that $j<n$. Then $P_n=I_j$ is a tilting module and it contains $L_n$. Therefore
$P_n=T_n$. As $\star$ is a simple preserving duality, we have $T_n^{\star}=T_n$ and therefore
$P_n=T_n=T_n^{\star}=I_j^{\star}=P_j$, which is a contradiction. This shows that 
$P_n$ is not injective.

We prove the claim of the theorem by induction on $n$. By the above, the basis of the induction
is the case $n=2$. We always have $T_1=L_1$. As $P_1=I_1$ is projective-injective, it is a tilting module.
If $L_1=T_1=P_1=I_1$, then, by the same argument as in the second paragraph of the proof we get a 
contradiction using the fact that $\Lambda$ is connected. Therefore $P_1=T_2$. In particular,
$\Delta_2$  embeds into $T_2=I_1$ and hence $\mathrm{soc}(\Delta_2)=L_1$.

So, the module $\Delta_2$ has simple top $L_2$ and simple socle $L_1$. If the composition 
multiplicity of $L_1$ in $\Delta_2$ were greater than $1$, that would imply existence of
a non-split self-extension for the simple module $L_1$. But since 
$\Lambda$ is assumed to be quasi-hereditary
no such self-extension can exist.
This means
that there is a short exact sequence
\begin{displaymath}
0\to L_1\to  \Delta_2\to L_2\to 0.
\end{displaymath}
As $P_2=\Delta_2$, it follows that the quiver of $\Lambda$ has exactly one arrow from the
vertex of $L_2$ to the vertex of $L_1$, call it $\alpha$. 
Applying $\star$, we see that there is exactly one arrow in the opposite direction,
call it $\beta$. As the composition multiplicity of $L_2$ in $P_2=\Delta_2$ is $1$, we have
$\alpha\beta=0$. At the same time, the composition multiplicity of $L_2$ in $P_2=I_2=T_1$ is 
at least $2$ as $P_2$ has a copy of  $L_2$ at the top, $I_2$ has a copy of  $L_2$ in the socle
and $T_2$ has a copy of $L_1$ somewhere. It follows that $\beta\alpha\neq 0$ and
we obtain that $\Lambda\cong A_2$. This completes the basis of our induction.

To prove the induction step, we assume that $n\geq 3$. Similarly to the case $n=2$,
we have that $T_1=L_1$ is not projective.
Therefore $L_1,P_1,P_2,\dots,P_{n-1}$ is a complete list of representatives of isomorphism
classes of tilting $\Lambda$-modules.  In particular, $T_n=P_j$, for some $j<n$.
The latter means that $\Delta_n=P_n$ embeds into $P_j$ and hence 
$P_n$ has simple socle $L_j$. Also, $\mathrm{Hom}_{\Lambda}(P_n,P_s)=0$, for all 
${s\in\{1,2,\dots,j-1,j+1,\dots,n-1\}}$ since for such $s$ we have $P_s=T_i$ for some $i<n$ and the composition multiplicity of $L_n$ in such tilting modules is zero. This implies that the quiver of $\Lambda$ contains
no arrows from the vertex of $L_n$ to the vertex of such $L_s$. By applying $\star$,
we see that there are no arrows back either. Therefore the only arrows from the vertex of 
$L_n$ in the quiver of $\Lambda$ are those going to the vertex of $L_j$ and the  only arrows to the vertex of  $L_n$ in the quiver of $\Lambda$ are those going from the 
vertex of $L_j$.

The equality $\mathrm{Hom}_{\Lambda}(P_n,P_s)=\mathrm{Hom}_{\Lambda}(P_n,I_s)=0$ 
also means that $P_n$ has no composition subquotients $L_s$ with $s$ as in the previous paragraph.
We claim that this implies that $j=n-1$. Indeed, assume $j<n-1$ and let $e$ be the
primitive idempotent of $\Lambda$ corresponding to $L_n$. Consider the quasi-hereditary
algebra $\Lambda/(\Lambda e\Lambda)$. As $L_n$ does not appear in $P_{n-1}$, the
module $P_{n-1}$ is a module over $\Lambda/(\Lambda e\Lambda)$. As $n-1$ is now maximal
and $\Lambda/(\Lambda e\Lambda)$ is quasi-hereditary, the module $P_{n-1}$ must be
a standard module for $\Lambda/(\Lambda e\Lambda)$. If $P_{n-1}=I_{n-1}=L_{n-1}$,
we obtain the same contradiction with connectedness of $\Lambda$ as above.
If $P_{n-1}=I_{n-1}$ is not simple, then the multiplicity of $L_{n-1}$ in $P_{n-1}$
must be at least two. This is not possible for a standard module.
Therefore $j=n-1$.

Similarly to the case $n=2$, we see that there is exactly one arrow from the
vertex of $L_n$ to the vertex of $L_{n-1}$, call it $\alpha$. 
Applying $\star$, we see that there is exactly one arrow in the opposite direction,
call it $\beta$. As the composition multiplicity of $L_n$ in $P_n=\Delta_n$ is $1$, we have
$\alpha\beta=0$. Furthermore the precomposition of $\alpha$ with any other arrow in the
quiver of $\Lambda$ is zero and, similarly, the postcomposition of $\beta$
with any other arrow in the quiver of $\Lambda$ is zero.

Now consider the basic quasi-hereditary 
algebra $\Lambda/(\Lambda e\Lambda)$. Its quiver is obtained from the
quiver of $\Lambda$ by deleting the vertex of $L_n$ and the arrows $\alpha$ and $\beta$.
Therefore $\Lambda/(\Lambda e\Lambda)$ is connected as $\Lambda$ was connected and
the vertex of $L_n$ only had arrows to and from the vertex of $L_{n-1}$ by the previous paragraph. The simple preserving duality on
$\Lambda$ induces a simple preserving duality on $\Lambda/(\Lambda e\Lambda)$. The 
projective-injective modules $P_1,P_2,\dots,P_{n-2}$ of $\Lambda$ are
also modules over $\Lambda/(\Lambda e\Lambda)$ and  remain 
projective-injective. Therefore $\Lambda/(\Lambda e\Lambda)$ satisfies all
assumptions of our theorem and has one less projective than $\Lambda$. 
Therefore we can now apply the induction assumption.

By induction, we have $\Lambda/(\Lambda e\Lambda)\cong B_{n-1}$. Also, from the above
we see that the quiver of $\Lambda$ is isomorphic to the quiver for $B_{n}$.
From what we concluded in the induction step we have that there are exact sequences
\begin{displaymath}
0\to \Delta_n\to P_{n-1}\to \Delta_{n-1}\to 0,
\end{displaymath}
\begin{displaymath}
0\to L_{n-1}\to \Delta_{n}\to L_{n}\to 0,\qquad 
0\to L_{n-2}\to \Delta_{n-1}\to L_{n-1}\to 0.
\end{displaymath}
As $P_{n-1}=I_{n-1}$ has simple socle and 
\begin{displaymath}
\mathrm{Ext}^{1}_{\Lambda}(L_n,L_{n-2})=\mathrm{Ext}^{1}_{\Lambda}(L_{n-2},L_{n})=0
\end{displaymath}
by the above, we obtain that $P_{n-1}$ has simple top and socle isomorphic to $L_{n-1}$
and $\mathrm{rad}(P_{n-1})/\mathrm{soc}(P_{n-1})=L_n\oplus L_{n-2}$. This means that 
in the part
\begin{displaymath}
\xymatrix{ 
L_{n-2}\ar@/^/[r]^{\delta}&L_{n-1}\ar@/^/[r]^{\beta}\ar@/^/[l]^{\gamma}&L_{n}\ar@/^/[l]^{\alpha} 
} 
\end{displaymath}
of the quiver of $\Lambda$ we have the relation $\gamma\delta=c\beta\alpha$, for some non-zero $c\in\Bbbk$.
Changing $\beta$ to $\beta'=c\beta$ defines an  isomorphism from $\Lambda$ to $B_{n}$.
This completes the proof.
\end{proof}

\section{Non-vanishing of $\mathrm{Ext}^2_{B_n}$}\label{s4}

\subsection{The result}\label{s4.1}

The aim of this subsection is to prove the following technical result.

\begin{proposition}\label{prop3}
Let $X$ be an indecomposable 
$B_n$-module. Assume that $X$ is neither projective nor in $\mathbf{D}$. Then
$\mathrm{Ext}^2_{B_n}(X,X)\neq0$.
\end{proposition}

Since we have a complete classification of indecomposable $B_n$-modules, to prove Proposition~\ref{prop3},
we need only to consider the cases when $X$ is either a simple module $L(i)$, with $i>1$, or one of the modules $W(i,j)$, $M(i,j)$, $N(i,j)$, $S(i,j)$ from Subsection~\ref{s2.4}, with $j>i+1$.

To simplify notation, in the following sections,
for a subset $Z\subset\{0,1,2,\dots,n\}$, we set 
\begin{displaymath}
P_Z=\bigoplus_{z\in Z}P(z). 
\end{displaymath}
For two integers $i\leq j$, we denote by $[i,j]$ 
the set of all integers $s$ satisfying the conditions $i\leq s\leq j$ and $s\equiv i\mod 2$.

\subsection{Proof of Proposition~\ref{prop3} for modules $W(i,j)$}\label{s4.2}
Let $e_k$ be the idempotent corresponding to $L(k)$ and let $J=B_neB_n$, where $e=e_{j+1}+\dots + e_n$. Then, by Lemma 2.2 in \cite{KK}, for any modules $X,Y\in B_n/J-\mathrm{mod}$ we have 
\begin{displaymath}
\mathrm{Ext}^k_{B_n/J}(X,Y)=\mathrm{Ext}^k_{B_n}(X,Y).
\end{displaymath}
This allows us to assume $j=n$ without losing any generality.

From the definition of $W(i,n)$, there is a surjection $P_{[i,n]}\tto W(i,n)$.
It is easy to see that the kernel of this surjection is isomorphic to $W(i-1,n-1)$, if $i\neq 1$,
and to $N(1,n-1)$, if $i=1$.

Consider first the case $i=1$. In this case there is a surjection from $P_{[2,n-1]}$ to $N(1,n-1)$
and it is easy to see that the kernel of this surjection is isomorphic to $N(2,n)$. The obvious
injection $N(2,n)\hookrightarrow W(1,n)$ does not factor through $P_{[2,n-1]}$ as
the latter maps only to the radical of $W(1,n)$ while the image of the injection 
$N(2,n)\hookrightarrow W(1,n)$ is not contained in the radical. This implies that the second self-extension of $W(1,n)$
does not vanish.

Now, consider the case $i=2$. In this case there is a surjection from $P_{[1,n-1]}$ to $W(1,n-1)$
and it is easy to see that the kernel of this surjection is isomorphic to $N(1,n)$. The obvious
projection $N(1,n)\tto W(2,n)$ does not factor through $P_{[1,n-1]}$ as
the latter maps only to the radical of $W(2,n)$. This implies that the second self-extension 
of $W(1,n)$ does not vanish.

Finally, consider the case $i>2$. In this case there is a surjection from $P_{[i-1,n-1]}$ to 
$W(i-1,n-1)$ and it is easy to see that the kernel of this surjection is isomorphic to $W(i-2,n)$. 
The obvious projection $W(i-2,n)\tto W(i,n)$ does not factor through $P_{[i-1,n-1]}$ as
the latter maps only to the radical of $W(i,n)$. This implies that the second self-extension 
of $W(i,n)$ does not vanish.

\subsection{Proof of Proposition~\ref{prop3} for modules $M(i,j)$}\label{s4.3}

Similarly to Subsection~\ref{s4.2}, we may assume $n=j$ without losing any generality.

From the definition of $M(i,n)$,  there is a surjection $P_{[i+1,n-1]}\tto M(i,n)$.
It is easy to see that the kernel of this surjection is isomorphic to $M(i+1,n-1)$, if $i<n-2$,
and to $L(n-1)$, if $i=n-2$.

Consider first the case $i=n-2$. In this case the kernel of $P(n-1)\tto L(n-1)$ is isomorphic to
$W(n-2,n)$. The dimension of the homomorphism space from $W(n-2,n)$ to $M(n-2,n)$ is two,
where one of the homomorphisms factors through $L(n)$ and the other one factors through $L(n-2)$.
At the same time, the dimension of the homomorphism space from $P(n-1)$ to $M(n-2,n)$ is one.
This implies that the second self-extension  of $M(n-2,n)$ does not vanish.

Consider now the case $i=n-4$. In this case, by recursion, the kernel of the surjection
$P(n-2)\tto M(n-3,n-1)$ is isomorphic to $L(n-2)$. The obvious injection $L(n-2)\hookrightarrow M(n-4,n)$ 
does not factor through $P(n-2)$ as the only map from $P(n-2)$ to $M(n-4,n)$ sends the
radical of $P(n-2)$ to zero. This implies that the second self-extension  of $M(n-4,n)$ does not vanish.

Finally, consider now the case $i<n-4$. In this case, by recursion, the kernel of 
$P_{[i+2,n-2]}\tto M(i+1,n-1)$ is isomorphic to $M(i+2,n-2)$. The obvious injection 
$M(i+2,n-2)\hookrightarrow M(i,n)$ does not factor through $P_{[i+2,n-2]}$
as the image of $M(i+2,n-2)$ in $P_{[i+2,n-2]}$ covers the socle of $P_{[i+2,n-2]}$
and any map from $P_{[i+2,n-2]}$ to $M(i,n)$ sends the socle of $P_{[i+2,n-2]}$ to zero.
This implies that the second self-extension  of $M(i,n)$ does not vanish.

\subsection{Proof of Proposition~\ref{prop3} for modules $N(i,j)$}\label{s4.4}

Similarly to Subsection~\ref{s4.2}, we may assume $n=j$ without losing any generality.
Note that $i<n-2$ as $N(n-1,n)\in\mathbf{D}$ and we must have $i\not\equiv n \mod 2$.

From the definition of $N(i,n)$, it follows that there is a surjection $P_{[i+1,n]}\tto N(i,n)$
and it is easy to see that the kernel of this surjection is isomorphic to $N(i+1,n-1)$. By recursion, there is a surjection $P_{[i+2,n-1]}\tto N(i+1,n-1)$
and it is easy to see that the kernel of this surjection is isomorphic to $N(i+2,n)$.
The obvious injection 
$N(i+2,n)\hookrightarrow N(i,n)$ does not factor through $P_{[i+2,n-1]}$ 
as the image of $N(i+2,n)$ in $P_{[i+2,n-1]}$ covers the socle of $P_{[i+2,n-1]}$
and any map from $P_{[i+2,n-1]}$ to $N(i,n)$ sends the socle of $P_{[i+2,n-1]}$ to zero.
This implies that the second self-extension  of $N(i,n)$ does not vanish.

\subsection{Proof of Proposition~\ref{prop3} for modules $S(i,j)$}\label{s4.5}

Similarly to Subsection~\ref{s4.2}, we may assume $n=j$ without losing any generality.
Note that $i<n-1$ as $S(n-1,n)\in\mathbf{D}$.

From the definition of $S(i,n)$, it follows that there is a surjection $P_{[i,n-1]}\tto S(i,n)$
and it is easy to see that the kernel of this surjection is isomorphic to $M(1,n-1)$, if $i=1$,
and to $S(i-1,n-1)$, if $i>1$.

Consider first the case $i=1$ and $n=4$. There is a surjection $P(2)\tto M(1,3)$ 
whose kernel is isomorphic to $L(2)$. The obvious injection $L(2)\hookrightarrow S(1,4)$
does not factor through $P(2)$ as the image of $L(2)$ in $P(2)$ is in the socle and
any map from $P(2)$ to $S(1,4)$ kills the socle of $P(2)$.
This implies that the second self-extension  of $S(1,4)$ does not vanish.

Consider now the case $i=1$ and $n>4$. There is a surjection $P_{[2,n-2]}\tto M(1,n-1)$ 
with kernel isomorphic to $M(2,n-2)$. The obvious injection $M(2,n-2)\hookrightarrow S(1,n)$
does not factor through $P_{[2,n-2]}$ as the image of $M(2,n-2)$ in $P_{[2,n-2]}$ intersects the socle and
any map from $P_{[2,n-2]}$ to $S(i,n)$ kills the socle of $P_{[2,n-2]}$.
This implies that the second self-extension  of $S(1,n)$ does not vanish.

Consider now the case $i=2$. In this case there is a surjection $P_{[1,n-2]}\tto S(1,n-1)$ 
whose kernel is isomorphic to $M(1,n-2)$. The unique up to a scalar non-zero homomorphism
$M(1,n-2)\to S(2,n)$ does not factor through $P_{[1,n-2]}$ as the image of 
$M(1,n-2)$ in $P_{[1,n-2]}$ intersects the socle and
any map from $P_{[1,n-2]}$ to $S(2,n)$ kills the socle of $P_{[1,n-2]}$.
This implies that the second self-extension  of $S(2,n)$ does not vanish.

Finally, consider the case $i>2$. In this case, by recursion, 
there is a surjection $P_{[i-1,n-2]}\tto S(i-1,n-1)$ 
whose kernel is isomorphic to $S(i-2,n-2)$. The unique up to a scalar non-zero homomorphism
$S(i-2,n-2)\to S(i,n)$ does not factor through $P_{[i-1,n-2]}$ as the image of 
$S(i-2,n-2)$ in $P_{[i-1,n-2]}$ intersects the socle and
any map from $P_{[i-1,n-2]}$ to $S(i,n)$ kills the socle of $P_{[i-1,n-2]}$.
This implies that the second self-extension  of $S(i,n)$ does not vanish.

\subsection{Proof of Proposition~\ref{prop3} for simple modules $L(i)$}
\label{s4.6}
Similarly to Subsection~\ref{s4.2} we can assume that $i=n$ without losing any generality.
The kernel of the projective cover $P(n)\tto L(n)$ is
$L(n-1)$ and $L(n)$ appears in the top of the kernel $P(n-1)\tto L(n-1)$. 
This implies that the second self-extension  of $L(n)$ does not vanish.

\section{Generalized tilting modules for $B_n$}\label{s5}

\subsection{Generalized tilting modules}\label{s5.1}

Let $\Lambda$ be a finite-dimensional algebra and $T$ a $\Lambda$-module. Recall from \cite{Mi}
that $T$ is called a {\em generalized tilting} $\Lambda$-module provided that 
\begin{itemize}
\item $T$ has finite projective dimension;
\item $\mathrm{Ext}^{i}_{\Lambda}(T,T)=0$, for all $i>0$;
\item the module $\Lambda_{\Lambda}$ has a finite coresolution by modules in $\mathrm{add}(T)$.
\end{itemize}

We will say that a $\Lambda$-module is {\em basic} provided that it is 
isomorphic to a direct sum of pairwise non-isomorphic indecomposable modules.

\subsection{Classification of basic generalized tilting $B_n$-modules}\label{s5.2}

\begin{theorem}\label{thm4}

Basic generalized tilting $B_n$-modules are, up to isomorphism, exactly the following modules:
\begin{equation}\label{eq1}
P(1)\oplus P(2)\oplus\dots\oplus P(n-1)\oplus X,
\end{equation}
where $X\in\mathbf{D}$.
\end{theorem}

\begin{proof}
Let $T$ be a generalized tilting $B_n$-module.
Corollary 1 in \cite{Mi} states that the number of isomorphism classes of simple modules in $B_n$ and $\mathrm{End}_{B_n}(T)$ are equal. It follows that the number of indecomposable summands of $T$ must be $n$.
From this it also follows that every projective-injective module must be a direct summand of any generalized tilting module. 
As $P(i)$ is both projective and injective for all $i<n$, 
we have that each such $P(i)$ is a summand of $T$.
Therefore $T$ is isomorphic to a module of the form \eqref{eq1}, for some $B_n$-module $X$.
As $T$ must be ext-self-orthogonal, from Proposition~\ref{prop3} it follows that 
$X\in\mathbf{D}$.

It remains to check that, for each $X\in\mathbf{D}$, the module \eqref{eq1} is a generalized
tilting module. As $B_n$ is quasi-hereditary, it has finite global dimension, so the projective
dimension of any module is finite. Further, from Proposition~\ref{cor1} and the fact that 
each $P(i)$, for $i<n$, is both projective and injective, it follows that the module
\eqref{eq1} is ext-self-orthogonal. What is left is to show that the module \eqref{eq1}
coresolves the right regular $B_n$-module. If $X$ is projective, the module \eqref{eq1} is equal to the right regular $B_n$-module. 
Note that to show that a module coresolves the right regular $B_n$-module it is enough to show that it coresolves every projective $B_n$-module. If $X$ is not projective, all projective $B_n$-modules except $P(n)$ is a direct summand of the module \eqref{eq1}. It therefore only remains to show that the module \eqref{eq1} coresolves $P(n)$.
The claim follows from Lemma~\ref{lem5} below 
since all projective-injective $B_n$-modules are direct summands of the module \eqref{eq1}.

\begin{lemma}\label{lem5}
Let $X\in\mathbf{D}$ be non-projective, then $X$ has a projective resolution of the form
\begin{displaymath}
0\to Q_k\to\dots \to Q_1\to Q_0\to X\to 0 
\end{displaymath}
where $Q_k\cong P(n)$ and all other $Q_i$ are projective-injective.
\end{lemma}

\begin{proof}
Write the modules in $\mathbf{D}\setminus\{\Delta(n)\}$ in the following order:
\begin{displaymath}
\Delta(n-1),\Delta(n-2),\dots,\Delta(1)=\nabla(1),\nabla(2),\dots,\nabla(n)
\end{displaymath}
and let us prove the claim by induction with respect to this order.

The basis of the induction follows from the short exact sequence
\begin{displaymath}
0\to P(n)\to P(n-1)\to\Delta(n-1)\to 0.
\end{displaymath}
The induction steps follow from the short exact sequences
\begin{displaymath}
0\to \Delta(i+1)\to P(i)\to\Delta(i)\to 0, \qquad i<n,
\end{displaymath}
and
\begin{displaymath}
0\to \nabla(i)\to P(i)\to\nabla(i+1)\to 0, \qquad i<n.
\end{displaymath}
\end{proof}

This completes the proof.
\end{proof}

\subsection{A partial order of generalized tilting $B_n$-modules}
\label{s5.3}
Let $\Lambda$ be an Artin algebra and let $\mathcal{T}_{\Lambda}$ denote the set of all basic generalized tilting $\Lambda$-modules up to isomorphism and define the \emph{right perpendicular category} of a $\Lambda$-module $T$ to be 
$$T^{\perp}:=\{X\in \Lambda\text{-mod} \; |\; \mathrm{Ext}_{\Lambda}^k(T,X)=0 \text{ for all } k>0\}.$$
We can now define a partial order $\leq$ on $\mathcal{T}_{\Lambda}$ by letting $T\leq T'$ if $T^{\perp}\subset T'^{\perp}$.  
In \cite{HU} the authors prove that, for any Artin algebra $\Lambda$, the Hasse diagram for $(\mathcal{T}_{\Lambda},\leq)$ is equal to the graph $\mathit{K}_{\Lambda}$ which is defined as follows. 
Let the vertices of $\mathit{K}_{\Lambda}$ be the elements of $\mathcal{T}_{\Lambda}$. There is an edge between $T$ and $T'$ if and only if $T=M\oplus X$, $T'=M\oplus Y$, where $X,Y$ is indecomposable and there exists a short exact sequence $$0\rightarrow X \rightarrow \tilde{M} \rightarrow Y \rightarrow 0$$ where $\tilde{M}\in \mathrm{add}(M)$.

By Theorem~\ref{thm4} every basic (generalized) tilting module is of the form $$P(1)\oplus P(2)\oplus \dots \oplus P(n-1)\oplus X,$$ for some $X\in \mathbf{D}$. So the set $\mathbf{D}$ is in bijection with the set $\mathcal{T}_{B_n}$. $\mathbf{D}$ therefore inherits a partial order from $\mathcal{T}_{B_n}$, which we will denote by $\prec$. From the short exact sequences
\begin{displaymath}
0\to \Delta(i+1)\to P(i)\to\Delta(i)\to 0, \qquad i<n,
\end{displaymath}
and
\begin{displaymath}
0\to \nabla(i)\to P(i)\to\nabla(i+1)\to 0, \qquad i<n.
\end{displaymath}
we then get the partial order

\begin{displaymath}
\nabla(n)\prec \dots \prec \nabla(2)\prec L(1) \prec \Delta(2)\prec \dots \prec \Delta(n).
\end{displaymath}

\section{Exceptional sequences for $B_n$}\label{s6}

\subsection{Exceptional sequences}\label{s6.1}

Let $\Lambda$ be a finite-dimensional algebra. Recall, see \cite{Bo},
that a $\Lambda$-module $M$ is called  \emph{exceptional} provided that
\begin{itemize}
\item $\mathrm{Hom}_{\Lambda}(M,M)=\Bbbk$;
\item $\mathrm{Ext}^i_{\Lambda}(M,M)=0$, for all $i>0$.
\end{itemize}
A sequence $\mathbf{M}=(M_1,\dots, M_k)$ of $\Lambda$-modules is called
an \emph{exceptional sequence} provided that 
\begin{itemize}
\item each $M_i$ is exceptional;
\item $\mathrm{Ext}^i_{\Lambda}(M_t,M_s)=0$, for all $1\leq s< t\leq k$ and all $i\geq 0$.
\end{itemize}
An exceptional sequence is called \emph{full (or complete)} if it generates the derived category.

\subsection{Classification of exceptional sequences for $B_n$}\label{s6.2}

\begin{proposition}\label{prop7}
A $B_n$-module $M$ is exceptional if and only if $M\in\mathbf{D}$.
\end{proposition}

\begin{proof}
Each of the modules $P(1),P(2),\dots,P(n-1)$ has a nilpotent endomorphism given by sending the
top to the socle. Therefore none of these modules is exceptional. If a $B_n$-module $M$ is
neither projective-injective nor in $\mathbf{D}$, then $M$ is not exceptional by 
Proposition~\ref{prop3}.

On the other hand, every $M\in\mathbf{D}$ has trivial endomorphism algebra and hence is exceptional
by Proposition~\ref{cor1}. The claim follows.
\end{proof}

\begin{theorem}\label{thm8}
There are exactly $2^{n-1}$ full exceptional sequences of $B_n$-modules and they are all of the the form
\begin{equation}\label{eq2}
(\nabla(i_1),\dots, \nabla(i_k),L(1),\Delta(j_1),\dots,\Delta(j_l)), 
\end{equation}
where
\begin{itemize}
\item
$k+l=n-1$, where $0\leq k,l \leq n-1$;
\item
$\{i_1,\dots,i_k,j_1,\dots, j_l\}=\{2,\dots,n\}$;
\item 
$i_s > i_t$ and $j_s < j_t$ for $s<t$.
\end{itemize}  
\end{theorem}

\begin{proof}
By Proposition~\ref{prop7}, any exceptional sequence of $B_n$-modules consists of modules from $\mathbf{D}$.

For any $i\in\{2,3,\dots,n\}$, we have $\Delta(i)\not\cong\nabla(i)$ and there are non-zero maps
$\Delta(i)\to\nabla(i)$ and $\nabla(i)\to\Delta(i)$. Therefore $\Delta(i)$ and $\nabla(i)$ cannot
appear in the same exceptional sequence. Consequently, any full exceptional sequence must be
of the form \eqref{eq2}, up to ordering of elements in this sequence. That this ordering of elements yields an exceptional sequence follows by Proposition~\ref{cor1} and the fact that there are no non-zero maps from $\Delta(i)$ to $\nabla(j)$ for $i\neq j$; from $\Delta(j)$ to $\Delta(i)$ and from $\nabla(i)$ to $\nabla(j)$ for $i<j$. 
The fact that this ordering
is uniquely defined follows from the following lemma.

\begin{lemma}\label{lem9}
Let $1\leq i<j\leq n$ and $k=j-i$. Then
\begin{displaymath}
\mathrm{Ext}^k_{B_n}(\Delta(i),\Delta(j))\neq0\quad\text{ and }\quad
\mathrm{Ext}^k_{B_n}(\nabla(j),\nabla(i))\neq0.
\end{displaymath}
\end{lemma}

\begin{proof}
There is an exact sequence:
\begin{displaymath}
0\to \Delta(j)\to  P(j-1)\to\dots P(i+1)\to P(i)\to \Delta(i)\to 0.
\end{displaymath}
This is the first part of the projective resolution of $\Delta(i)$. Note that the only projective modules with non-zero homomorphisms to $\Delta(j)$ are $P(j)$ and $P(j-i)$.
Furthermore, the identity map on $\Delta(j)$ does not factor through $P(j-1)$ as the embedding 
from $\Delta(j)\hookrightarrow  P(j-1)$ ends up in the radical of $P(j-1)$.
This implies the inequality $H^k(\mathrm{Hom}_{B_n}(P_{\bullet},\Delta(j)))=\mathrm{Ext}^k_{B_n}(\Delta(i),\Delta(j))\neq0$,
and the second inequality $\mathrm{Ext}^k_{B_n}(\nabla(j),\nabla(i))\neq0$ follows using the simple preserving duality defined in Section~\ref{s2.3}.
\end{proof}

It remains to check that any sequence of the form \eqref{eq2} generates the derived category.
For this it is enough to show that all simple modules can be obtained from modules in 
\eqref{eq2} using the operations of taking kernels of epimorphisms and cokernels of
monomorphisms. Let us prove, by induction on $i$, that $L(i)$ can be obtained in this way.
The basis of the induction is trivial as $L(1)$ is in \eqref{eq2}. Now let us prove the
induction step. Assume that $L(i-1)$ can be obtained. Then \eqref{eq2} contains either 
$\Delta(i)$ or $\nabla(i)$. In the first case, $L(i)$ is the cokernel of the inclusion
$L(i-1)\hookrightarrow \Delta(i)$. In the second case, $L(i)$ is the kernel of the projection
$\nabla(i)\tto L(i-1)$. The claim of the theorem follows.
\end{proof}

\section{Generalized tilting modules and exceptional sequences for $C_n$}\label{s7}

\subsection{Self-orthogonal $C_n$-modules}\label{s7.1}

\begin{proposition}\label{prop21}
The only ext-self-orthogonal $C_n$-modules are the projective-injective modules,
and also the modules $N(i,i+1)$ and $S(i,i+1)$, where $i=1,2,\dots,n-1$.
\end{proposition}

\begin{proof}
That the projective-injective modules are ext-self-orthogonal is clear, this also applies to
 the projective modules $N(n-1,n)$ and $S(1,2)$ and the
injective modules $N(1,2)$ and $S(n-1,n)$.

For $i<n-1$, the module $N(i,i+1)$ has the following projective resolution:
\begin{displaymath}
0\to N(n-1,n)\to P(n-1)\to\dots \to P(i+2)\to P(i+1)\to N(i,i+1)\to 0. 
\end{displaymath}
The only term in this resolution that has a non-zero homomorphism to
$N(i,i+1)$ is the zero term $P(i+1)$. This implies that $$H^k(\mathrm{Hom}_{C_n}(P_{\bullet},N(i,i+1)))=\mathrm{Ext}^k_{C_n}(N(i,i+1),N(i,i+1))=0$$ for all $k>0$. Therefore $N(i,i+1)$ is an
ext-self-orthogonal module.

The algebra $C_n$ has the obvious automorphism $\alpha$ that swaps $1$ and $n$,
$2$ and $n-1$ and so on. Twisting by $\alpha$ swaps modules
of type $S$ and modules over type $N$. Therefore ext-self-orthogonality
of $S(i,i+1)$ follows from ext-self-orthogonality of $N(n-i,n+1-i)$
by applying $\alpha$.

It remains to show that all other $C_n$-modules are not ext-self-orthogonal.
To start with, let us show that all simple $C_n$-modules are not 
ext-self-orthogonal. If $i\neq 1,n$, then the kernel of the projective 
cover $P(i)\tto L(i)$ is $W(i-1,i+1)$. The latter 
has projective cover  $P(i-1)\oplus P(i+1)\tto W(i-1,i+1)$ whose
kernel has $L(i)$ at the top. Therefore $\mathrm{Ext}^2_{C_n}(L(i),L(i))\neq 0$.
Similarly, the kernel of the projective cover $P(1)\tto L(1)$ is
$L(2)$ and $L(1)$ appears in the top of the kernel $P(2)\tto L(2)$.
Therefore $\mathrm{Ext}^2_{C_n}(L(1),L(1))\neq 0$. Applying 
$\alpha$, we get $\mathrm{Ext}^2_{C_n}(L(n),L(n))\neq 0$.
So, all $C_n$-simples are not ext-self-orthogonal.

Now, for $X\in\{W,M,S,N\}$, $1\leq i$ and $i+1<j\leq n$, consider the module $X(i,j)$.
If $i\geq 4$, then the first two terms of the projective resolution of
$X(i,j)$ are exactly the same modules as for $B_n$ and therefore the fact that 
$X(i,j)$ is not ext-self-orthogonal follows from Proposition~\ref{prop3}.
Applying $\alpha$, we obtain that $X(i,j)$ is not ext-self-ortho\-go\-nal
if $j\leq n-3$. So, it remains to consider the nine modules $X(i,j)$, where
$i=1,2,3$ and $j=n-2,n-1,n$. Taking into account the parity of $n$,
we have altogether $18$ different cases left.

Consider the case $n$ is odd and $X(i,j)=W(1,n)$. Then the kernel of
the projective cover $P_{[1,n]}\tto W(1,n)$ is isomorphic to 
$W(2,n-1)$, if $n>3$, and $L(2)$, if $n=3$. 
The kernels of the projective covers $P_{[2,n-1]}\tto W(2,n-1)$ and $P_2\tto L(2)$
are isomorphic to $W(1,n)$ in both cases. The identity homomorphism of $W(1,n)$
does not factor through $P_{[2,n-1]}$ and hence defines a
non-zero element in $\mathrm{Ext}^2_{C_n}(W(1,n),W(1,n))$.
Therefore $W(1,n)$ is not ext-self-orthogonal.

Consider the case $n$ is even and $X(i,j)=S(1,n)$. Then the kernel of
the projective cover $P_{[1,2]}\tto S(1,n)$ is isomorphic to 
$S(2,n-1)$. The kernel of the projective cover 
$P_{[2,n-2]}\tto S(2,n-1)$
is isomorphic to $S(1,n-2)$. The obvious inclusion of $S(1,n-2)$
to $S(1,n)$ does not factor through $P_{[2,n-2]}$ and hence defines a
non-zero element in $\mathrm{Ext}^2_{C_n}(S(1,n),S(1,n))$.
Therefore $S(1,n)$ is not ext-self-orthogonal.

The remaining cases are similar and left to the reader. The claim 
of the proposition follows.
\end{proof}

\subsection{Generalized tilting $C_n$-modules}\label{s7.2}

In this subsection we assume $n>2$ as the algebra $C_2$ is self-injective.

\begin{theorem}\label{thm22}
Basic generalized tilting $C_n$-modules are, up to isomorphism,
exactly the following modules:
\begin{equation}\label{eq3}
P(2)\oplus P(3)\oplus\dots\oplus P(n-1)\oplus S(i,i+1)\oplus N(j,j+1), 
\end{equation}
where $i,j\in\{1,2,\dots,n-1\}$.
\end{theorem}

\begin{proof}
As mentioned in the proof of Theorem \ref{thm4} the basic projective-injective module
$P(2)\oplus P(3)\oplus\dots\oplus P(n-1)$ must be a summand of any 
generalized tilting module and the total number of summands must be $n$. 
This leaves us with two undetermined summands which, by Proposition \ref{prop21}, must be of the form $N(i,i+1)$ or $S(i,i+1)$, where $i=1,2,\dots,n-1$.
For $1\leq s<t\leq n-1$, we have the
exact sequence
\begin{equation}\label{eq4}
0\to N(t,t+1)\to P(t)\to P(t-1)\to\dots\to P(s+1)\to N(s,s+1). 
\end{equation}
For $t=n-1$, this is the projective resolution of $N(s,s+1)$ since $N(n-1,n)=P(n)$. For $s<t<n$ the sequence \eqref{eq4} is the first part of the projective resolution of $N(s,s+1)$.
Note that the only projective modules with non-zero homomorphisms to $N(t,t+1)$ are $P(t)$ and $P(t+1)$.
Furthermore, the identity endomorphism on $N(t,t+1)$ does not factor through
$P(t)$ which implies that $H^{t-s}(\mathrm{Hom}_{C_n}(P_{\bullet},N(t,t+1)))\neq 0$ and hence gives a non-zero extension from $N(s,s+1)$
to $N(t,t+1)$ of degree $t-s$. Therefore two modules of type $N$
cannot be summands of the same generalized tilting module.

Applying the automorphism $\alpha$ to the sequence \eqref{eq4}, for $1\leq s<t\leq n-1$, we get the exact sequence
\begin{equation}\label{eq42}
0\to S(s,s+1)\to P(s+1)\to\dots\to P(t-1)\to P(t)\to S(t,t+1). 
\end{equation}
For $s=1$ this is the projective resolution of $S(t,t+1)$ since $S(1,2)=P(1)$.
We also obtain a non-zero extension from $S(t,t+1)$ to $S(s,s+1)$ of degree $t-s$ by applying $\alpha$ to the corresponding result for modules of type $N$. 
Hence two 
modules of type $S$
cannot be summands of the same generalized tilting module.
Therefore from Proposition~\ref{prop21} it follows that
any generalized tilting $C_n$-module is of the form \eqref{eq3}.

It remains to prove that every module of the form \eqref{eq3}
is a generalized tilting module. From \eqref{eq4} and \eqref{eq42} it follows that any module of
the form \eqref{eq3} has finite projective dimension and coresolves the regular $C_n$-module. 
More precisely, if the module \eqref{eq3} is projective, it is equal to the regular $C_n$-module. 
If the module \eqref{eq3} is not projective, all projective $C_n$-modules except $P(1)$ and $P(n)$ is a direct summand of the module \eqref{eq3}. That the module \eqref{eq3} coresolves these modules follows from the exact sequence \eqref{eq4} with $t=n-1, s=j$ and the exact sequence \eqref{eq42} with $s=1,t=i$.

So, it remains to check that any 
module of of the form \eqref{eq3} is ext-self-orthogonal, more
precisely, that, for $k>0$, we have 
\begin{equation}\label{eq5}
\mathrm{Ext}^k_{C_n}(S(i,i+1),N(j,j+1))=0=
\mathrm{Ext}^k_{C_n}(N(j,j+1),S(i,i+1)).
\end{equation}
Using $\alpha$, the first equality reduces to the second one.
Consider the projective resolution of $N(j,j+1)$ which is
given by \eqref{eq4}, for $s=j$ and $t=n-1$. If $i<j$, none of the terms in this
resolution has a non-zero homomorphism to $S(i,i+1)$ and hence all
extensions from  $N(j,j+1)$ to $S(i,i+1)$ vanish in this case.

If $i=j$, then only the term in homological degree zero of the resolution
has a non-zero homomorphism to $S(i,i+1)$ and hence all extensions of
non-zero homological degree vanish. If $i>j$, then only terms in homological
degrees $i-j$ and $i-j-1$ have non-zero homomorphisms to $S(i,i+1)$.
Consider the corresponding truncations of the projective resolutions.
The first truncation is \eqref{eq4} for $s=j$ and $t=i$. The unique 
non-zero homomorphism from $N(i,i+1)$ to $S(i,i+1)$ does factor through
$P(i)$ and hence the extension at this position vanishes. 
The second truncation is \eqref{eq4} for $s=j$ and $t=i-1$. In this
case there are no non-zero homomorphisms from $N(i-1,i)$ to $S(i,i+1)$
and hence there is no extension at this position either.
This prove \eqref{eq5} and completes the proof.
\end{proof}

\subsection{Exceptional sequences for $C_n$}\label{s7.3}

\begin{theorem}\label{thm23}
There are no full exceptional sequences of $C_n$-modules.
\end{theorem}

\begin{proof}
As projective-injective modules do not have trivial endomorphism algebra,
from Proposition~\ref{prop21} it follows that any exceptional sequence
of modules must consist of modules of the form $S(i,i+1)$ and $N(i,i+1)$.
As
\begin{displaymath}
\mathrm{Hom}_{C_n}(S(i,i+1),N(i,i+1))\neq 0\quad\text{ and }\quad 
\mathrm{Hom}_{C_n}(N(i,i+1),S(i,i+1))\neq 0,
\end{displaymath}
the modules $S(i,i+1)$ and $N(i,i+1)$ cannot belong to the same exceptional
sequence of $C_n$-modules. Therefore an exceptional sequence of
$C_n$-modules cannot contain more than $n-1$ modules and thus 
cannot be full.
\end{proof}

\vspace{5mm}

\noindent

Department of Mathematics, Uppsala University, Box. 480,
SE-75106, Uppsala, SWEDEN, email: {\tt elin.persson.westin\symbol{64}math.uu.se}

\end{document}